\documentclass[10pt]{article}

\usepackage{amsmath, amsthm, amssymb, amsfonts, MnSymbol, hyperref, graphicx} 

 \newtheorem{thm}{Theorem}
 \newtheorem{lem}[thm]{Lemma}
 \newtheorem{cor}[thm]{Corollary}
 \newtheorem{prop}[thm]{Proposition}
 \newtheorem{rmk}{Remark}

\newcommand{\field}[1]{\mathbb{#1}}

\begin{document}

\author{J. C. Meyer and D. J. Needham}

\title{The evolution to localized and front solutions in a non-Lipschitz reaction-diffusion Cauchy problem with trivial initial data}

\maketitle

\begin{abstract}
In this paper, we establish the existence of spatially inhomogeneous classical self-similar solutions to a non-Lipschitz semi-linear parabolic Cauchy problem with trivial initial data. Specifically we consider bounded solutions to an associated two-dimensional non-Lipschitz non-autonomous dynamical system, for which, we establish the existence of a two-parameter family of homoclinic connections on the origin, and a heteroclinic connection between two equilibrium points. Additionally, we obtain bounds and estimates on the rate of convergence of the homoclinic connections to the origin.

\end{abstract}

%Keywords
%semi-linear parabolic PDE\sep self-similar\sep non-Lipschitz \sep homoclinic connection \sep heteroclinic connection

%\MSC 35K58\sep 34C37  

\section{Introduction}
\label{sec1}

In this paper, we study classical bounded solutions $u:\field{R}\times [0,T]\to\field{R}$ to the non-Lipschitz semi-linear parabolic Cauchy problem 
\begin{equation} \label{i1} u_t - u_{xx} = u|u|^{p-1} \ \ \ \text{ on } \field{R}\times (0,T] , \end{equation}
\begin{equation} \label{i2} u=0\ \ \  \text{ on } \field{R}\times \{ 0 \} , \end{equation}
with $0<p<1$ and $T>0$ (which we henceforth refer to as [CP]). The primary achievement of the paper is the establishment of the existence of a two-parameter family of localized spatially inhomogeneous solutions to [CP] for which {\mbox{$u(x,t)\to 0$}} as $|x|\to\infty$ uniformly for $t\in [0,T]$; the secondary achievement of the paper is the establishment of front solutions to [CP], which approach $\pm (1-p)^{1/(1-p)}t$ as $|x|\to \pm \infty$ uniformly for $t\in [0,T]$. We note here that for $p\geq 1$ in \eqref{i1}, then the unique bounded classical solution with initial data \eqref{i2} is the trivial solution, see for example {\mbox{\cite[Theorem 4.5]{JMDN1}.}} 

Qualitative properties of non-negative (non-positive) solutions to \eqref{i1} when {\mbox{$0<p<1$,}} with non-negative (non-positive) initial data, and for which $u(x,t)$ is bounded as $|x|\to\infty$ uniformly for $t\in [0,T]$, have been determined in \cite{JAME1}, \cite{ACKDJN1}, \cite{DN1}, \cite{PMMJALDJN1} and \cite{JMDN2}. However, we note that any non-negative (non-positive) classical bounded solution to [CP] must be spatially homogeneous for $t\in [0,T]$, see for example \cite[Corollary 2.6]{JAME1}. Thus, the solutions constructed in this paper are two signed on $\field{R}\times [0,T]$. The authors are currently unaware of any studies of two signed solutions to \eqref{i1}-\eqref{i2} with $0<p<1$. Generic local results for spatial homogeneity of solutions to semi-linear parabolic Cauchy problems with homogeneous initial data depend upon uniqueness results, see for example, \cite{JMDN2}. For results concerning the related problem of asymptotic homogeneity (in general, asymptotic symmetry) as $t\to \infty$ of non-negative (non-positive) global solutions to semi-linear parabolic Cauchy problems, we refer the reader to the survey article \cite{PP1}. 

Non-negative (non-positive), spatially inhomogeneous solutions to \eqref{i1} for $p>1$ have been considered in \cite{AHFW1}, \cite{FW1} \cite{FW2}, \cite{DNPGC1}, \cite{MEOK1}, \cite{MHEY1}, \cite{PPEY1}, \cite{CDMH1}, \cite{TCFDFW2} and \cite{NSKW1} with the focus primarily on critical exponents for finite time blow-up of solutions, and conditions for the existence of global solutions (see the review articles \cite{HL1} and \cite{KDHL1}). Moreover, for $p>1$, solutions to \eqref{i1} with two signed initial data have been considered in \cite{NMEY1} and \cite{NMEY2}, whilst boundary value problems have been studied in \cite{JB1} and \cite{TCFDFW1}. 

The paper is structured as follows; in Section 2 we introduce the self-similar solution structure for [CP], and hence, determine an ordinary differential equation related to \eqref{i1}; the remainder of the paper concerns the study of particular solutions to this ordinary differential equation, which is re-written as an equivalent two-dimensional non-autonomous dynamical system. Specifically, in Section 3 we establish the existence of a two-parameter family of homoclinic connections on the equilibrium $(0,0)$. Additionally, we determine bounds and estimates on the asymptotic approach of these solutions to $(0,0)$. In Section 4, we establish the existence of a heteroclinic connection between the equilibrium points {\mbox{$(\pm (1-p)^{1/(1-p)},0)$}}. 

\section{Self-Similar Structure}
With $0<p<1$ and $T>0$, we refer to $u:\field{R}\times [0,T]\to\field{R}$ as a solution to [CP] when $u$ satisfies \eqref{i1}-\eqref{i2} with regularity,
\begin{equation} \label{2a} u\in L^{\infty}(\field{R}\times [0,T])\cap C(\field{R}\times [0,T])\cap C^{2,1}(\field{R}\times (0,T]) . \end{equation}
Observe that $u^\pm:\field{R}\times [0,T]\to\field{R}$ given by 
\[ u^\pm(x,t)= \pm ((1-p)t)^{1/(1-p)} \ \ \ \forall (x,t)\in\field{R}\times [0,T] \]
are the maximal and minimal solutions to [CP] (see \cite[Chapter 8]{JMDN1}), and hence any solution $u:\field{R}\times [0,T] \to\field{R}$ to [CP] must satisfy,
\begin{equation} \label{2b} u^-(x,t) \leq u(x,t) \leq u^+(x,t)\ \ \   \forall (x,t)\in\field{R}\times [0,T] .\end{equation}
To construct spatially inhomogeneous solutions to [CP], we consider, for any fixed $x_0\in\field{R}$, self-similar solutions $u:\field{R}\times [0,T]\to\field{R}$ of the form,
\begin{equation} \label{2c} u(x,t) = \begin{cases} w\left(\frac{x-x_0}{t^{1/2}} \right) t^{1/(1-p)} &, (x,t) \in \field{R}\times (0,T], \\ 0 &, (x,t) \in \field{R} \times \{ 0 \} , \end{cases} \end{equation}
with $w:\field{R}\to\field{R}$ to be determined. Now, $u:\field{R}\times [0,T]\to\field{R}$ given by \eqref{2c} is a solution to [CP] if and only if there exist constants $\alpha$, $\beta\in\field{R}$ such that $w:\field{R}\to\field{R}$ satisfies the following zero-value problem, namely, 
\begin{align}
\label{z1} & w'' + \frac{1}{2}\eta w' + w|w|^{p-1} - \frac{1}{(1-p)} w = 0 \ \ \ \forall \eta \in \field{R} ,\\
\label{z2} & w(0) = \alpha ,\ \ \ w'(0) = \beta ,\\
\label{z3} & w\in C^2(\field{R})\cap L^\infty (\field{R}) .
\end{align}
Here $\eta = (x-x_0)/t^{1/2}$, and we observe that the ordinary differential equation \eqref{z1} is both non-autonomous and non-Lipschitz. It is convenient to introduce 
\[ x=w,\ \ \ y=w',\]
after which the problem \eqref{z1}-\eqref{z3} is equivalent to the zero-value problem for the two-dimensional, non-Lipschitz, non-autonomous, dynamical system, 
\begin{align}
\label{3} & x'= y \\
\label{4} & y'= \frac{1}{(1-p)}x - x|x|^{p-1} - \frac{1}{2}\eta y \ \ \ \forall \eta \in\field{R}, \\
\label{z3'} & (x(0),y(0)) = (\alpha , \beta ), \\
\label{z3''} & (x,y)\in C^1(\field{R})\cap L^\infty (\field{R}) .
\end{align}
We refer to the equivalent zero-value problems in \eqref{z1}-\eqref{z3} and \eqref{3}-\eqref{z3''} as (S). Our objective is now to investigate those $(\alpha , \beta )\in \field{R}^2$ for which (S) has a non-trivial solution. It is instructive to note, at this stage, via \eqref{2b}, that we may conclude that any solution to (S) must satisfy the inequality,
\begin{equation} \label{2b'} -(1-p)^{\frac{1}{(1-p)}} \leq w(\eta ) \leq (1-p)^{\frac{1}{(1-p)}} \ \ \ \forall \eta \in\field{R} ,\end{equation}
whilst, following \cite[Corollary 2.6]{JAME1}, any non-constant solution to (S) must be two-signed in $w$.

\section{Homoclinic Connections} \label{hom}
In this section we establish the existence of a two parameter family of homoclinic connections for (S) on the equilibrium point $(0,0)$ of the dynamical system \eqref{3}-\eqref{4}, and establish decay rates to the equilibrium point $(0,0)$ as $|\eta |\to\infty$ on these homoclinic connections. 
\subsection{Existence}
In this subsection, we establish the existence of homoclinic connections attached to the equilibrium point {\mbox{$(x,y)=(0,0)$}} of the dynamical system \eqref{3}-\eqref{4}. To begin, observe that ${\bf{Q}}:\field{R}^3\to\field{R}^2$, where
\begin{equation} \label{5} {\bf{Q}} (x,y,\eta ) = (Q_1,Q_2)(x,y,\eta )= \left( y, \frac{1}{(1-p)}x - x|x|^{p-1} - \frac{1}{2} \eta y \right) \ \ \ \forall (x,y,\eta )\in \field{R}^3 \end{equation}
is such that ${\bf{Q}}\in C(\field{R}^3)$, but also that ${\bf{Q}}$ is not locally Lipschitz continuous on $\field{R}^3$ (note that ${\bf{Q}}$ is locally Lipschitz continuous on $\field{R}^3 \backslash N$, with $N$ any neighbourhood of the plane $x=0$). We now have,

\begin{thm} \label{thm1}
The problem (S) with zero-value $(\alpha , \beta )\in\field{R}^2$ has a solution for $\eta \in [-\delta , \delta ]$ (not necessarily unique), where $\delta = 1/ (1+M)$ and 
\[ M= \max_{(x,y,\eta )\in R} |{\bf{Q}} (x,y,\eta )| \]
with 
\[ R= \{ (x,y,\eta )\in\field{R}^3 : |x-\alpha | \leq 1 ,\ |y-\beta |\leq 1 ,\ |\eta |\leq 1 \} .\]
\end{thm}

\begin{proof}
This follows immediately from the Cauchy-Peano Local Existence Theorem (see \cite[Chapter 1, {\mbox{Theorem 1.2}}]{CODLEV1}) since ${\bf{Q}}:\field{R}^3\to\field{R}^2$ is such that {\mbox{${\bf{Q}}\in C(\field{R}^3 )$.}}
\end{proof}

\begin{rmk} \label{rmk2}
When $\alpha \not= 0$, then the solution to (S) with zero-value $(\alpha , \beta )\in \field{R}^2$ is unique for $\eta \in [ -\delta ' , \delta ' ]$ for some $0 < \delta ' \leq \delta $. In addition,  the problem (S) with zero-value \mbox{$(\pm (1-p)^{1/(1-p)} , 0 )$} has the unique global solution
\begin{equation} \label{18*} (x(\eta ) , y(\eta )) = (\pm (1-p)^{1/(1-p)} , 0 ) \ \ \ \forall \eta \in \field{R}.\end{equation}
This follows since ${\bf{Q}}$ is locally Lipschitz in a neighbourhood of $(\pm (1-p)^{1/(1-p)},0)$ respectively. Also, the problem (S) with zero-value (0,0) has the unique global solution,
\[ (x(\eta ),y(\eta )) = (0,0) \ \ \ \forall \eta \in \field{R} . \]
In this case uniqueness does not follows immediately, since ${\bf{Q}}$ is not locally Lipschitz continuous in any neighborhood of $(0,0)$, but instead follows after further qualitative results have been established for solutions to (S) (see Remark \ref{rmk6}). 
\end{rmk}
\noindent We now introduce the function $V:\field{R}^2\to\field{R}$ defined by,
\begin{equation} \label{7} V(x,y) = \frac{1}{2} y^2 - \frac{1}{2(1-p)} x^2 +\frac{1}{(1+p)}|x|^{1+p} \ \ \ \forall (x,y)\in\field{R}^2. \end{equation}
We observe immediately that 
\begin{equation} \label{8} V\in C^{1,1}(\field{R}^2) ,\end{equation}
with 
\begin{equation} \label{10} \nabla V(x,y) = \left( \frac{-1}{(1-p)} x + x|x|^{p-1} , y \right) \ \ \ \forall (x,y)\in\field{R}^2 . \end{equation}
We now examine  the structure of the level curves of $V$ in $\field{R}^2$, namely, the family of curves in $\field{R}^2$ defined by
\begin{equation} \label{9} V(x,y) = c , \end{equation}
for $-\infty < c <\infty$. It is straightforward to establish that the family of level curves of $V$ are qualitatively as sketched in Figure 1, with $\cal{H}$ representing the two level curves connecting $(-(1-p)^{1/1-p},0)$ to $((1-p)^{1/1-p},0)$ and enclosing the origin. In Figure \ref{fig1}, on the red curve $V=(1-p)^{2/(1-p)}/(2(1+p))$, whilst on the blue curves $V=0$. At $(\pm (1-p)^{1/(1-p)} ,0)$ then $V=(1-p)^{2/(1-p)}/(2(1+p))$, whilst at $(0,0)$ then $V=0$. Inside $\cal{H}$, the level curves are simple closed curves concentric with the origin $(0,0)$, and $V$ is increasing from $V=0$ at the origin $(0,0)$, as each level curve is crossed, when moving out from the origin $(0,0)$ to the boundary curve $\cal{H}$, on which $V=(1-p)^{2/(1-p)}/(2(1+p))$. Thus, inside $\cal{H}$, $V$ has a minimum at the origin $(0,0)$ and is increasing on moving radially away from the origin $(0,0)$ to the boundary $\cal{H}$. On the level curves exterior and above or below $\cal{H}$, then $V>(1-p)^{2/(1-p)}/(2(1+p))$, whilst on the level curves to the left and right side of $\cal{H}$, then $V<(1-p)^{2/(1-p)}/(2(1+p))$, with $V=0$ on the blue level curves.  
\begin{figure}
  \centering
    \includegraphics[scale=0.2]{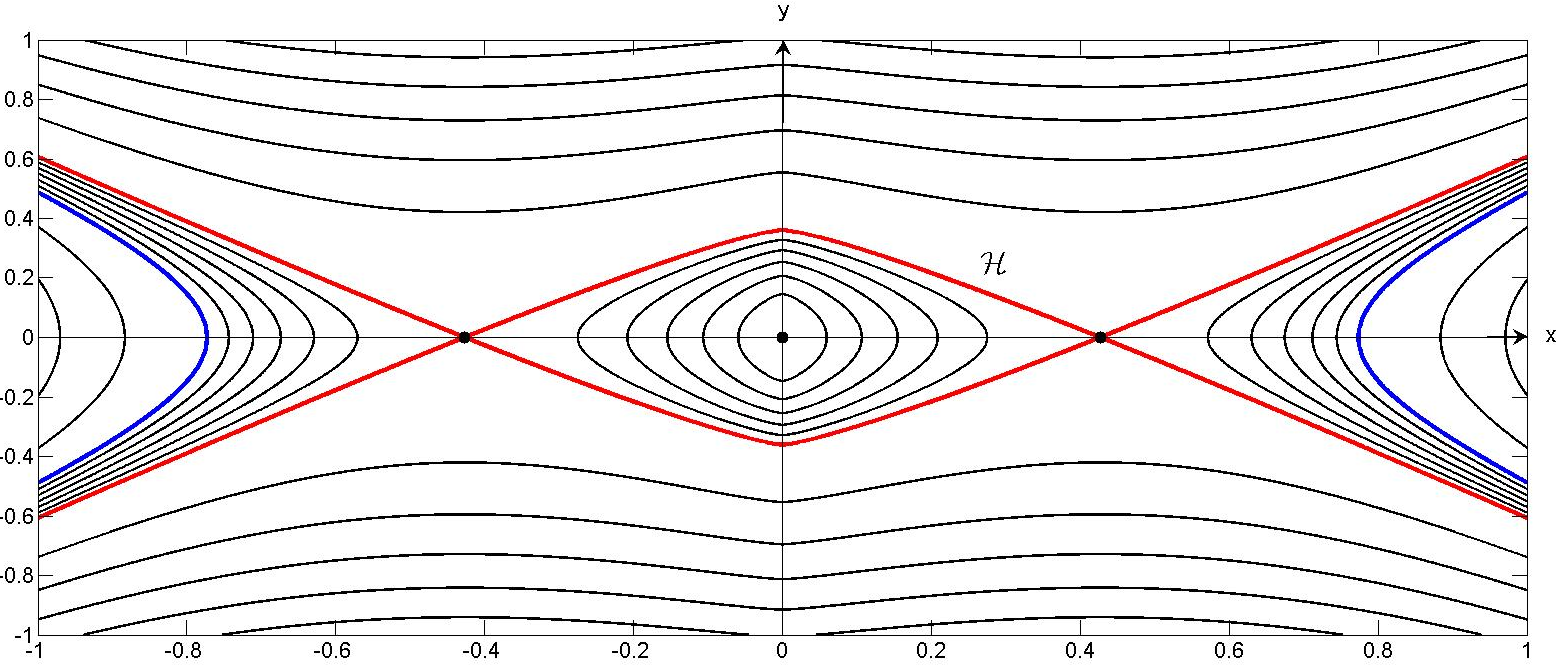}  \label{fig1}
		\caption{A qualitative sketch of the level curves of $V$}
\end{figure}
We now focus on the level curves of $V$ on and inside $\cal{H}$, which have
\begin{equation} \label{12} 0 \leq c \leq c^*(p) ,\end{equation}
where 
\begin{equation} \label{13} c^*(p)=\frac{(1-p)^{2/(1-p)}}{2(1+p)} . \end{equation}
These are concentric closed curves surrounding the origin $(0,0)$. We will label the interior of the level curve $V=c$ by $D_c$, with the level curve $V=c$ labelled as $\partial D_c$, for $0 \leq c \leq c^*(p)$. In addition, we label the set 
\[ \bar{D}_{c^*(p)}' = \bar{D}_{c^*(p)} \backslash \{ (\pm (1-p)^{1/(1-p)} , 0) , (0,0)\} .\]
Now let $(x^*(\eta ) , y^*(\eta ))$ be any solution to (S) for $\eta \in [-E,E]$ (any $E>0$) with zero-value $(\alpha , \beta )\in \field{R}^2$, and define $F:[-E,E]\to\field{R}$ as,
\begin{equation} \label{14} F(\eta ) = V(x^*(\eta ) , y^*(\eta )) \ \ \ \forall \eta \in [-E,E]. \end{equation}
Then $F\in C^1([-E,E])$, and via \eqref{3}, \eqref{4} and \eqref{5}, 
\begin{align*}
F'(\eta )  = & \nabla V (x^*(\eta ),y^*(\eta )) . (x^{*'}(\eta ) , y^{*'} (\eta ))  \\
 = & \nabla V (x^*(\eta ),y^*(\eta )) . {\bf{Q}}(x^*(\eta ) , y^*(\eta ) , \eta ) \ \ \ \forall \eta \in [ -E,E].
\end{align*}
It then follows, via \eqref{10} and \eqref{5} that,
\begin{equation} \label{15} F'(\eta ) = -\frac{1}{2} \eta (y^*(\eta ))^2\ \ \ \forall\eta \in [-E,E]. \end{equation}
It follows from \eqref{15} that
\begin{align}
\label{16} & F(\eta ) \text{ is non-increasing for }\eta \in [0,E],\\
\label{17} & F(\eta ) \text{ is non-decreasing for }\eta \in [-E,0].
\end{align}
We can now establish the following,

\begin{lem} \label{thm3}
Let $(x^* (\eta ),y^*(\eta ))$ be any solution to (S) on $[-E,E]$ (any $E>0$) with zero-value $(\alpha , \beta )\in \bar{D}_{c^*(p)}'$. Then 
\[ (x^*(\eta ) , y^*(\eta ))\in D_c \ \ \ \forall \eta \in [-E,E]\backslash \{ 0 \}, \]
where $c=V(\alpha , \beta )$.
\end{lem}

\begin{proof}
Let the zero-value $(\alpha , \beta )\in \partial D_c\backslash \{ \pm ((1-p)^{\frac{1}{1-p}} ,0) \}$ with \[ 0<c=V(\alpha , \beta ) \leq c^*(p). \] We first consider the case when $\beta \not= 0$. It follows from \eqref{15}-\eqref{17} that, 
\begin{equation} \label{18} F(\eta ) < F(0) \ \ \ \forall \eta\in [-E,E]\backslash \{ 0 \} . \end{equation}
Therefore, via \eqref{18}, 
\[ V(x^*(\eta ), y^*(\eta )) <c \ \ \ \forall \eta \in [-E,E] \backslash \{ 0 \} ,\]
and so 
\[ (x^*(\eta ) , y^*(\eta ))\in D_c \ \ \ \forall \eta\in [-E,E] \backslash \{ 0 \} ,\]
as required. Now consider the case when $\beta = 0$. Then $0<|\alpha |< (1-p)^{1/(1-p)}$ and therefore, via \eqref{4} $y^{*'}(0) \not= 0$ after which a similar argument completes the proof.
\end{proof}
\noindent We now have:

\begin{thm} \label{thm4}
For each $(\alpha , \beta )\in \bar{D}_{c^*(p)}'$, then (S) with zero-value $(\alpha , \beta )$ has a solution $(x^*(\eta ), y^*(\eta ))$ on $[-E,E]$ (any $E>0$). Moreover, every such solution satisfies $(x^*(\eta ), y^*(\eta ))\in D_c$ for all $\eta \in [-E,E] \backslash \{ 0 \}$, where $c=V(\alpha , \beta )$.
\end{thm}

\begin{proof}
For any $(\alpha , \beta )\in\bar{D}_{c^*(p)}'$, Lemma \ref{thm3} establishes that (S) with zero-value $(\alpha , \beta )$ is a priori bounded. The result then follows by a finite number of applications of the Cauchy-Peano Local Existence Theorem (see \cite[Chapter 1, Theorem 1.2]{CODLEV1}), with $\delta = 1/(1+M)$ and
\[ M= \max_{(x,y,\eta )\in R'} |{\bf{Q}}(x,y,\eta )| \]
whilst
\[ R'= \{ (x,y,\eta )\in\field{R}^3 : |x| \leq 2(1-p)^{1/(1-p)} ,\ |y |\leq 2\sqrt{2c^*(p)} ,\ |\eta |\leq 2E \} .\]
The final statement follows immediately from Lemma \ref{thm3}.
\end{proof}

\noindent We can now establish a global existence result for (S), namely

\begin{cor} \label{cor5}
For $(\alpha , \beta) \in \bar{D}_{c^*(p)}'$ then (S) with zero-value $(\alpha , \beta )$ has a solution $(x^*(\eta ), y^*(\eta ))$ on $\field{R}$. Moreover, every such solution satisfies $(x^*(\eta ), y^*(\eta ))\in D_c$ for all $\eta \in \field{R} \backslash \{ 0\}$, where $c=V(\alpha , \beta )$. 
\end{cor}

\begin{proof}
Since Theorem \ref{thm4} holds for any $E>0$, the result follows immediately. 
\end{proof}

\begin{rmk} \label{rmk6}
Let $(x^*(\eta ) , y^*(\eta ))$ be any solution to (S) on $[-E,E]$ with zero-value $(0,0)$. It follows from \eqref{7}, \eqref{14} and \eqref{15} that
\begin{equation} \label{20} V(x^*(\eta ),y^*(\eta ))=F(\eta ) \leq F(0) = V(0,0) = 0 \ \ \ \forall \eta \in [-E,E]. \end{equation} 
Thus $(x^*(\eta ),y^*(\eta ))\in \cal{S}$ for all $\eta \in [-E,E]$, with $\cal{S}$ being a connected subset of 
\[ \{ (x,y) \in\field{R}^2 : V(x,y) \leq 0 \}\]
for which $(0,0)\in \cal{S}$. It follows that ${\cal{S}} = \{ (0,0) \}$ and so $(x^*(\eta ),y^*(\eta )) = (0,0) $ for all $\eta \in [-E,E]$. We conclude that the unique solution to (S) with zero-value $(0,0)$ is given by,
\[ (x^*(\eta ),y^*(\eta )) = (0,0) \ \ \ \forall \eta \in \field{R} .\]
\end{rmk}

We next introduce the function $H:\field{R}\to\field{R}$ such that
\begin{equation} \label{21} H(x) = \frac{1}{(1-p)} x - x|x|^{p-1}\ \ \ \forall x\in\field{R} , \end{equation}
and observe that 
\begin{equation} \label{22} H\in C(\field{R}) . \end{equation}
We have,

\begin{lem} \label{lem7}
Let $(\alpha , \beta ) \in \bar{D}_{c^*(p)}'$, and let $(x^*(\eta ), y^*(\eta ))$ for $\eta \in \field{R}$ be a global solution to (S) with zero-value $(\alpha , \beta )$. Then 
\[ y^*(\eta ) \to 0 \text{ as } |\eta |\to\infty . \]
\end{lem}

\begin{proof}
We establish the result for $\eta \to \infty$; the result for $\eta\to -\infty$ follows similarly. Now, from \eqref{4},
\begin{equation} \label{23} y^{*'}(\eta )  =H(x^*(\eta )) - \frac{1}{2}\eta y^*(\eta ) \ \ \ \forall \eta \in [0, \infty ) .\end{equation}
It then follows from \eqref{23} that,
\begin{equation} \label{24} y^*(\eta ) = \beta e^{-\frac{1}{4}\eta^2} +  e^{-\frac{1}{4}\eta^2}\int_0^\eta H(x^*(s)) e^{\frac{1}{4}s^2}ds \ \ \ \forall \eta \in [ 0,\infty ) .\end{equation}
Thus,
\begin{equation} \label{25} |y^*(\eta )| \leq |\beta | e^{-\frac{1}{4}\eta^2} +  e^{-\frac{1}{4}\eta^2} \int_0^\eta | H(x^*(s)) | e^{\frac{1}{4}s^2}ds \ \ \ \forall \eta \in [0,\infty ) .\end{equation}
However, via Corollary \ref{cor5}, $(x^*(\eta ), y^*(\eta ))\in \bar{D}_{c^*(p)}$ for $\eta \in [0,\infty )$, and so, via \eqref{22}, there exists a constant $M_H \geq 0$ such that 
\begin{equation} \label{26} |H(x^*(s))| \leq M_H \ \ \ \forall s \in [0,\infty ).\end{equation}
It then follows from \eqref{25} and \eqref{26} that 
\begin{equation} \label{27} |y^*(\eta )| \leq |\beta | e^{-\frac{1}{4}\eta^2} +  M_H e^{-\frac{1}{4}\eta^2} \int_0^\eta e^{\frac{1}{4}s^2}ds \ \ \ \forall \eta \in [0,\infty ).\end{equation} 
Now a simple application of Watson's Lemma (see \cite[Proposition 2.1]{PDM1}), gives,
\begin{equation} \label{28} \int_0^\eta e^{\frac{1}{4}s^2}ds \sim \frac{2}{\eta} e^{\frac{1}{4} \eta^2} \text{ as } \eta \to \infty . \end{equation}
We then have, via \eqref{27} and \eqref{28}, that 
\begin{equation} \label{29} |y^*(\eta )| \leq |\beta | e^{-\frac{1}{4}\eta^2} +  \frac{4M_H}{\eta } \text{ as } \eta\to\infty .  \end{equation} 
It follows from \eqref{29} that 
\[ y^*(\eta )\to 0 \text{ as } \eta\to \infty , \]
as required
\end{proof}

\noindent We next have,

\begin{lem} \label{lem8}
Let $(x^*(\eta ) , y^* (\eta ))$ for $\eta \in \field{R}$ be a global solution to (S) with zero-value $(\alpha , \beta )\in \bar{D}_{c^*(p)}'$, and $F:\field{R}\to\field{R}$ as in \eqref{14}. Then $F(\eta )$ is non-increasing for $\eta \in (0,\infty )$ and non-decreasing for $\eta \in (-\infty , 0)$, with 
\[ F(\eta ) \to \begin{cases} F_\infty & \text{ as }\eta \to\infty \\F_{-\infty} & \text{ as } \eta \to -\infty \end{cases} \]
where $ F_\infty , F_{-\infty} \in [0, F(0))$.
\end{lem}

\begin{proof}
We observe from Corollary \ref{cor5} that 
\begin{equation} \label{30} (x^*(\eta ) , y^*(\eta ))\in D_c \ \ \ \forall \eta \in \field{R} \backslash \{ 0\}, \end{equation}
with $c=V(\alpha , \beta )=F(0)$, and so, 
\begin{equation} \label{31} 0 \leq F(\eta ) < F(0)\ \ \ \forall \eta \in \field{R} \backslash \{ 0\} . \end{equation}
In addition, it follows from \eqref{31}, \eqref{16} and \eqref{17}, since $F\in C^1(\field{R})$, that there exist \mbox{$F_\infty , F_{-\infty}\in\field{R}$,} such that
\[ F(\eta ) \to \begin{cases} F_\infty & \text{ as }\eta \to\infty \\F_{-\infty} & \text{ as } \eta \to -\infty \end{cases} \]
where $ F_\infty , F_{-\infty} \in [0, F(0))$, as required. 
\end{proof}

\noindent We now have,

\begin{thm} \label{thm9}
Let $(x^*(\eta ) , y^* (\eta ))$ for $\eta \in \field{R}$ be a global solution to (S) with zero-value $(\alpha , \beta )\in \bar{D}_{c^*(p)}'$. Then,
\[ (x^*(\eta ) , y^*(\eta )) \to (0,0) \text{ as } |\eta |\to\infty . \]
\end{thm}

\begin{proof}
We establish the result for $\eta \to \infty$. The result for $\eta \to -\infty$ follows similarly. We first recall from Corollary \ref{cor5} that,
\begin{equation} \label{32} (x^*(\eta ) , y^*(\eta ))\in D_{c^*(p)} \ \ \ \forall \eta \in \field{R} \backslash \{ 0 \} ,\end{equation}
and from Lemma \ref{lem7} that,
\begin{equation} \label{33} y^*(\eta )\to 0 \text{ as } \eta \to \infty . \end{equation}
In addition, we have from Lemma \ref{lem8} that,
\begin{equation} \label{34} V(x^*(\eta ),y^*(\eta ))\to F_\infty \text{ as } \eta \to\infty \end{equation}
for some $ F_\infty \in [0, c^*(p))$. It follows from \eqref{32}-\eqref{34} that 
\begin{equation} \label{36} x^*(\eta ) \to  x_\infty \text{ or } x^*(\eta ) \to -x_\infty \text{ as }\eta \to \infty \end{equation}
where $x_\infty$ is the single non-negative root of 
\[ V(x,0) = F_\infty \text{ with } x\in [ 0, (1-p)^{1/(1-p)} ). \]
Without loss of generality we will suppose that 
\begin{equation} \label{38} (x^*(\eta ) , y^*(\eta ))\to (x_\infty , 0)\text{ as } \eta \to\infty . \end{equation}
However, it follows from \eqref{4} that, 
\begin{equation} \label{39} y^*(\eta ) = \beta e^{-\frac{1}{4}\eta^2} + e^{-\frac{1}{4} \eta^2} \int_0^\eta H(x^*(s)) e^{\frac{1}{4}s^2} ds \ \ \ \eta \in [0,\infty )\end{equation}
with $H:\field{R}\to\field{R}$ given by \eqref{21}, and
\begin{equation} \label{40} H(x_\infty ) \leq 0 . \end{equation}
Using \eqref{36}, it is straightforward to establish that, when, 
\begin{equation} \label{41} H(x_\infty ) <0 , \end{equation}
then from \eqref{39},
\begin{equation} \label{42} y^*(\eta ) \sim \frac{2H(x_\infty )}{\eta } \text{ as } \eta \to \infty . \end{equation} 
In addition, from \eqref{3}, we have, 
\begin{equation} \label{43}  x^*(\eta ) = \alpha + \int_0^\eta y^*(s)ds \ \ \ \forall \eta \in [0,\infty ) , \end{equation} 
which gives, via \eqref{42}, that 
\[ x^*(\eta ) \sim 2 H(x_\infty ) \log{\eta},\ \ \  \text{ as } \eta\to\infty ,\]
which contradicts \eqref{36}. We conclude that \eqref{41} cannot hold, and so, via \eqref{40}, we must have
\begin{equation} \label{44} H(x_\infty ) = 0 , \end{equation}
which, since $x_\infty \in [0, (1-p)^{1/(1-p)})$, requires $x_\infty = 0$. It then follows from \eqref{38} that,
\[ (x^*(\eta ) , y^*(\eta )) \to (0,0) \text{ as } \eta \to \infty , \] 
as required.
\end{proof}

We conclude from Corollary \ref{cor5} and Theorem \ref{thm9} that the problem (S) has a two parameter family of nontrivial, distinct homoclinic connections on the equilibrium point $(0,0)$, parametrized by \mbox{$(\alpha , \beta)\in \bar{D}_{c^*(p)}'$} which we will denote by $w_{\alpha , \beta}:\field{R}\to\field{R}$ for each $(\alpha ,\beta )\in \bar{D}_{c^*(p)}'$. Here $w=w_{\alpha , \beta }(\eta )$, $\eta \in \field{R}$, has zero-values $w(0)=\alpha$, $w'(0)=\beta $. Moreover, 
\[ (w_{\alpha , \beta } (\eta ) , w_{\alpha , \beta}'(\eta ))\in D_{V(\alpha , \beta )} \ \ \ \forall \eta \in\field{R}\backslash \{ 0 \} . \]
Additionally, note that $w_{0,\beta }(\eta )$ is an odd function of $\eta$ whilst $w_{\alpha , 0}(\eta )$ is an even function of $\eta $. Furthermore, it also follows from the comments below \eqref{2b'} that $w_{\alpha , \beta }(\eta )$ must be two signed for $\eta\in\field{R}$.

\subsection{Decay Bounds and Estimates} \label{sec21}
In this section, we establish results concerning the rate of decay to zero of $w_{\alpha , \beta }(\eta )$ as $\eta \to\pm\infty$. Specifically, we establish algebraic bounds on the rate of decay of $w_{\alpha , \beta } (\eta )$ as $\eta \to \pm \infty $, and hence, determine that $w_{\alpha , \beta }\in L_q(\field{R})$ for each $q>(1-p)/2$. From these bounds we may infer that the corresponding solution to [CP], say {\mbox{$u_{\alpha , \beta}:\field{R}\times [0,\infty )\to\field{R}$,}} satisfies $u(\cdot , t)\in L_q(\field{R})$ for each $t\in [0,\infty )$ and $q>(1-p)/2$. To complement the algebraic bounds, we also provide a rational asymptotic approximation to the decay rate of $w_{\alpha , \beta}(\eta )$ as $\eta\to\pm \infty$, which, in fact suggests exponential decay as $\eta\to\pm\infty$.  

To begin, observe that $w=w_{\alpha , \beta }(\eta )$ for $\eta \in\field{R}$, via \eqref{z1}, satisfies
\[ (e^{\frac{1}{4}\eta^2}w')' = H(w)e^{\frac{1}{4}\eta^2} \ \ \ \forall \eta \in \field{R}.\]
It follows from two successive integrations, that
\begin{equation} \label{wd} w'(\eta )= \beta e^{-\frac{1}{4}\eta^2} + e^{-\frac{1}{4}\eta^2} \int_0^\eta H(w(s))e^{\frac{1}{4}s^2}ds \ \ \ \forall \eta \in\field{R} \end{equation}
whilst,
\begin{equation} \label{w} w(\eta )= \alpha + \int_0^\eta \beta e^{-\frac{1}{4}t^2} dt + \int_0^\eta e^{-\frac{1}{4}t^2} \int_0^t H(w(s))e^{\frac{1}{4}s^2}ds dt \ \ \ \forall \eta \in \field{R}. \end{equation}
We now have,

\begin{prop} \label{prop43}
Let $w:\field{R}\to\field{R}$ be a solution to (S) with zero-value $(\alpha, \beta )\in  \bar{D}_{c^*(p)}' $. Suppose that 
\[ |w(\eta )| \leq \frac{c_1}{(1+|\eta | )^\sigma} \ \ \ \forall \eta \in \field{R}.\]
with $\sigma \geq 0$ and $c_1 >0$ (independent of $\alpha$ and $\beta$). Then, there exists $c_2>0$, which depends on $c_1$, $\sigma$ and $p$, (independent of $\alpha$ and $\beta$) such that,
\[ |w'(\eta )| \leq \frac{c_2}{(1+|\eta | )^{\sigma p+1}} \ \ \ \forall\eta \in \field{R}.\]
\end{prop}

\begin{proof}
We give a proof for $\eta \geq 0$; the result for $\eta < 0$ follows similarly. Observe that 
\begin{equation} \label{Z6'} |H(w(\eta ))| = \left| \frac{1}{(1-p)}w(\eta ) - |w(\eta )|^{p-1}w(\eta )\right| \leq \frac{c_1^p}{(1+\eta )^{\sigma p}} \ \ \ \forall \eta \in [0,\infty ), \end{equation}
since, via Corollary \ref{cor5}, $|w(\eta )|< (1-p)^{\frac{1}{(1-p)}}$ for $\eta \in [0,\infty )$. Thus, via \eqref{wd} and \eqref{Z6'}, we have,
\begin{equation} \label{de3} |w'(\eta )|  \leq |\beta | e^{-\frac{1}{4}\eta^2} +  c_1^pe^{-\frac{1}{4}\eta^2} \int_0^\eta \frac{1}{(1+s)^{\sigma p}} e^{\frac{1}{4}s^2} ds \ \ \ \forall \eta \in [0,\infty ) .\end{equation}
Now, the second term on the right hand side of \eqref{de3} is a non-negative continuous function for $\eta\in [0,\infty )$, with asymptotic form, 
\[  c_1^p e^{-\frac{1}{4}\eta^2} \int_0^\eta \frac{1}{(1+s)^{\sigma p}} e^{\frac{1}{4}s^2} ds \sim \frac{2c_1^p}{\eta^{\sigma p +1}} \text{ as }\eta\to\infty .\]
It follows that,
\[  c_1^p e^{-\frac{1}{4}\eta^2} \int_0^\eta \frac{1}{(1+s)^{\sigma p}} e^{\frac{1}{4}s^2} ds \leq \frac{4c_1^p}{\eta^{\sigma p +1}} \text{ as }\eta\to\infty .\]
We conclude that there exists a positive constant $c_2$, depending upon $c_1$, $p$, and $\sigma$, such that 
\[  |w'(\eta )| \leq \frac{c_2}{(1+\eta )^{\sigma p +1}} \ \ \ \forall \eta \in [0,\infty ), \]
as required. 
\end{proof}

\noindent We next have,

\begin{prop} \label{prop43'}
Let $w:\field{R}\to\field{R}$ be a solution to (S) with zero-value $(\alpha, \beta )\in  \bar{D}_{c^*(p)}'$. Then,
\[ (w(\eta ),w'(\eta))\to (0,0)  \text{ as } \eta \to \pm\infty , \]
and moreover, 
\[ |w'(\eta )| \leq \frac{c_2}{(1+|\eta |) } \ \ \ \forall \eta \in \field{R}, \]
with $c_2>0$ dependent upon $p$ (independent of $\alpha$ and $\beta$). 
\end{prop}

\begin{proof}
The first conclusion follows directly from Theorem \ref{thm9}. Additionally, it follows from Corollary \ref{cor5} that 
\[ (w(\eta ),w'(\eta ))\in \mathcal{H} \ \ \  \forall \eta \in \field{R},\]
and hence, it follows from Proposition \ref{prop43} (with $\sigma = 0$, $c_1=(1-p)^{1/(1-p)}$) that 
\[ |w'(\eta )| \leq \frac{c_2}{(1+|\eta | ) } \ \ \ \forall \eta \in \field{R} , \]
as required. 
\end{proof}

\noindent We now demonstrate that every solution $w:\field{R}\to\field{R}$ to (S) with zero-value {\mbox{$(\alpha, \beta )\in  \bar{D}_{c^*(p)}'$}} decays to zero as $\eta \to \pm\infty$, with decay rate which is at least algebraic in $\eta$ as $\eta \to \pm\infty$. In particular, we demonstrate that $w:\field{R}\to\field{R}$ is contained in $L^{q} (\field{R})$ for any $q>(1-p)/2$. The proof is based on the decay bounds obtained in \cite{AHFW1}. 

\begin{thm} \label{thm44}
Let $w:\field{R}\to\field{R}$ be a solution to (S) with zero-value \mbox{$(\alpha, \beta )\in  \bar{D}_{c^*(p)}'$}. Then, for any $\epsilon >0$, there exists $c_{1\epsilon},c_{2\epsilon}>0$ (dependent generally on $\alpha$, $\beta$, $p$ and $\epsilon$) such that 
\[ |w(\eta )| < \frac{c_{1\epsilon}}{(1+|\eta | )^{\frac{2}{(1-p)} - \epsilon}} \ \ \ \forall \eta \in \field{R} , \]
\[ |w'(\eta )| < \frac{c_{2\epsilon}}{(1+|\eta | )^{\frac{(1+p)}{(1-p)} - \epsilon}} \ \ \ \forall \eta \in \field{R} . \]
\end{thm}

\begin{proof}
We give a proof for $\eta \geq 0$; the argument for $\eta < 0$ follows similarly. Observe on multiplying \eqref{z1} by $\eta^{-1}w(\eta )$, we have, 
\begin{equation} \label{51'} \frac{1}{\eta} \left[ |w(\eta )|^{1+p}- \frac{(w(\eta ))^2}{(1-p)}\right] = - \left[ \frac{(w(\eta ))^2}{4} + \frac{w(\eta )w'(\eta )}{\eta}\right]' + \frac{(w'(\eta ))^2}{\eta} - \frac{w(\eta )w'(\eta )}{\eta^2} \end{equation}
for $\eta \in (0,\infty )$. Additionally, via Proposition \ref{prop43'}, it follows that there exists $\eta_* \in (0,\infty )$ such that,  
\begin{equation} \label{120'} |w(\eta )| \leq \left(\frac{2p(1-p)}{(1+p)}\right)^{\frac{1}{(1-p)}} \ \ \ \forall\eta \in [\eta_*,\infty ) ,\end{equation}
and for $F:[0,\infty )\to\field{R}$ given by 
\[ F(\eta ) = V(w(\eta ), w'(\eta )) \ \ \ \forall \eta \in [0,\infty ),\]
that
\begin{equation} \label{120''} 0 \leq F(\eta ) \leq \left( \frac{4(c(p))^{\frac{2}{(1+p)}}}{C(p)} \right)^{(1+p)/(1-p)} \ \ \  \eta \in [\eta_*,\infty ) ,\end{equation}
where 
\begin{equation} \label{120'''} c(p) = \frac{1}{(1+p)} -\frac{1}{2},\text{ and }  C(p)= \frac{2(1+p)}{(1-p)} + 1 .\end{equation}
Thus, it follows from \eqref{51'} that 
\begin{align}
\nonumber \frac{F(\eta )}{\eta } &  = \frac{(w'(\eta ))^2}{2\eta } +  \frac{1}{\eta } \left[ -\frac{(w(\eta ))^2}{2(1-p) } + \frac{|w(\eta )|^{1+p}}{(1+p)}\right] \\
\nonumber & \leq \frac{(w'(\eta ))^2}{2\eta } +  \frac{1}{\eta } \left[ -\frac{(w(\eta ))^2}{(1-p) } + |w(\eta )|^{1+p} \right] \\
\label{Z8} & = \frac{3(w'(\eta ))^2}{2\eta} - \left[ \frac{(w(\eta ))^2}{4} + \frac{w(\eta )w'(\eta )}{\eta } \right]' - \frac{w(\eta )w'(\eta )}{\eta^2} , 
\end{align}
for $\eta \in [\eta_*,\infty )$. Since $F(\eta ) \geq 0$ for all $\eta \in [\eta^* , \infty )$, together with the decay estimates in Proposition \ref{prop43'}, it follows that we may integrate inequality \eqref{Z8} from $\eta\ (\geq \eta^* )$ to $l$, and then allow $l\to\infty$, to obtain,
\begin{equation} \label{Z9} \int_\eta^\infty \frac{F(t)}{t} dt \leq \frac{(w(\eta ))^2}{4} + \frac{2}{\eta} \sup_{t\geq \eta} |w(t)w'(t)| + \frac{3}{2} \int_\eta^\infty \frac{(w'(t))^2}{t} dt \end{equation}
for $ \eta \in [\eta_*,\infty )$. We also note, that since, via Corollary \ref{cor5}, $|w(\eta )|< (1-p)^{1/(1-p)}$, we have,
\begin{equation} \label{Z10} F(\eta ) \geq |w(\eta )|^{1+p}c(p) \geq 0, \end{equation}
for $\eta \in [\eta_*,\infty )$. It therefore follows from \eqref{Z9} and \eqref{Z10} that 
\begin{equation} \label{Z12}  0 \leq \int_\eta^\infty \frac{F(t )}{t }dt \leq \frac{1}{4}\left(\frac{F(\eta )}{c(p)} \right)^{\frac{2}{(1+p)}} + \frac{2}{\eta} \sup_{t\geq \eta}| w(t)w'(t)| + \frac{3}{2} \int_\eta^\infty \frac{(w'(t))^2}{t} dt \end{equation}
for $\eta \in [\eta_*, \infty )$. We observe that the right hand side of \eqref{Z12} is uniformly bounded for $\eta \in [\eta_*,\infty )$ via Proposition \ref{prop43'}.

Now suppose that there exists $k>0$ such that 
\begin{equation} \label{Z13} F(\eta ) \leq \frac{k}{\eta^\sigma}\ \ \ \forall \eta \in [\eta_*,\infty) \end{equation}
for some $\sigma \geq 0$ (note that \eqref{Z13} holds when $\sigma =0$ via Proposition \ref{prop43'}). Then, via \eqref{Z10}, it follows that there exists $c_1>0$ such that
\begin{equation} \label{126'} |w(\eta )| \leq \frac{c_1}{\eta^{\frac{\sigma}{(1+p)}}} \ \ \ \forall \eta \in [\eta_*,\infty ) \end{equation}
and so, via Proposition \ref{prop43}, there exists $c_2>0$ such that
\begin{equation} \label{126''} |w'(\eta )| \leq \frac{c_2}{\eta^{\frac{\sigma p}{(1+p)} + 1}}  \ \ \ \forall \eta \in [\eta_*,\infty ) .\end{equation}
Thus, it follows from \eqref{Z12}-\eqref{126''} and \eqref{120''}, that there exists $c_3,c_4,c_5>0$ such that
\begin{align}
\nonumber \int_\eta^\infty \frac{F(t )}{t }dt & \leq \frac{1}{4}\left(\frac{F(\eta )}{c(p)} \right)^{\frac{2}{(1+p)}} + \frac{c_3}{\eta^{\sigma +2}} + \frac{c_4}{\eta^{\frac{2\sigma p}{(1+p)} +2}} \\
\label{Z14} & \leq \frac{F(\eta )}{C(p)} + \frac{c_5}{\eta^{\frac{2\sigma p}{(1+p)} +2}}
\end{align}
for $\eta \in [\eta_*,\infty )$. Upon setting $G:[\eta_*,\infty )\to\field{R}$ to be
\[ G(\eta ) =  \int_\eta^\infty \frac{F(t )}{t }dt \ \ \ \forall \eta \in [\eta_*,\infty ), \]
it follows from \eqref{Z14} that $G$ satisfies,
\begin{equation} \label{Z15} \left( t^{C(p)} G(t)\right)'\leq \frac{c_6}{ t^{\frac{2\sigma p}{(1+p)} +3 -C(p)}} \ \ \ \forall t \in [\eta_*,\infty ), \end{equation}
with $c_6>0$ constant. An integration of \eqref{Z15} gives 
\begin{equation} \label{Z16} G(\eta ) \leq \frac{c_7}{\eta^{\frac{2\sigma p}{(1+p)} +2}} + \frac{c_8}{\eta^{C(p)}} \ \ \ \forall \eta \in [\eta_*,\infty ), \end{equation}
with $c_7,c_8>0$ constants. Also, recalling, via Lemma \ref{lem8}, that $F(\eta )$ is non-increasing on $[\eta^* , \infty )$, we have,  
\begin{equation} \label{Z17} G(\eta ) \geq \int_\eta^{2\eta} \frac{F(t)}{t}dt \geq \frac{1}{2} F(2\eta ), \ \ \ \forall \eta \in [\eta_*,\infty ) . \end{equation}
Thus, it follows from \eqref{Z16} and \eqref{Z17} that there exist constants $c_9,c_{10}>0$ such that 
\begin{equation} \label{Z18} F(\eta ) \leq \frac{c_9}{\eta^{\frac{2\sigma p}{(1+p)} +2}} + \frac{c_{10}}{\eta^{C(p)}} \ \ \ \forall \eta \in [\eta_*,\infty ). \end{equation}
Since \eqref{Z13} holds for $\sigma =0$, it follows that there exists sequences $\{\sigma_n\}_{n\in\field{N}}$ and  $\{k_n\}_{n\in\field{N}}$ given by 
\begin{equation} \label{Z19} \sigma_1=0,\ \ \ \sigma_{n+1}= \min\left\{ \frac{2\sigma_{n} p}{(1+p)} +2,\ C(p) \right\} \end{equation}
such that  
\begin{equation} \label{Z20} F(\eta ) \leq \frac{k_n}{\eta^{\sigma_n}}\ \ \ \forall \eta \in [\eta_*,\infty ). \end{equation}
We obtain from \eqref{Z19} and \eqref{120'''} that,
\[ \sigma_n= \frac{2(1+p)}{(1-p)} - \frac{4p}{(1-p)}\left(\frac{2p}{(1+p)}\right)^{n-2} \ \ \ \forall n\in\field{N} \]
and hence $\sigma_n$ is increasing with 
\begin{equation} \label{Z21} \sigma_n\to \frac{2(1+p)}{(1-p)} \ \ \ \text{ as } n\to\infty . \end{equation}
Therefore it follows, via \eqref{Z10} and \eqref{Z19}-\eqref{Z21}, that for each $\epsilon >0$ there exists $c_{1\epsilon} >0$ such that 
\begin{equation} \label{135} |w (\eta )| \leq \frac{c_{1\epsilon}}{(1+\eta)^{\frac{2}{(1-p)} -\epsilon}} \ \ \ \forall \eta \in [0,\infty ) ,\end{equation}
recalling that $w(\eta )$ is bounded on $[0,\eta^*]$. The bound on $|w'(\eta )|$ follows immediately from \eqref{135} and Proposition \ref{prop43}.
\end{proof}

The algebraic bounds in Theorem \ref{thm44} are the tightest decay rates we have been able to establish rigorously. However, the following asymptotic argument indicates that, in fact, $w=w_{\alpha , \beta }(\eta )$ decays exponentially in $\eta$ as $|\eta |\to\infty$, accompanied by rapid oscillatory behaviour. To this end, we now consider the asymptotic structure of $w=w_{\alpha , \beta } (\eta )$ as $\eta \to \infty$, with the same structure following as $\eta \to - \infty$. Now, for $\eta >> 1$, then $w=w_{\alpha , \beta }(\eta )$ satisfies, 
\begin{equation} \label{DI1} w'' + \frac{1}{2} \eta w' + w|w|^{p-1} - \frac{1}{(1-p)} w= 0 \ \ \ \eta >>1 \end{equation}
\begin{equation} \label{DI2} w(\eta ),\ w'(\eta ) \to 0 \ \ \ \text{ as }\eta \to\infty , \end{equation}
via \eqref{z1} and Proposition \ref{prop43'}. On using \eqref{DI2}, the dominant form of \eqref{DI1} when $\eta >>1$ is 
\begin{equation} \label{d3} w'' + w|w|^{p-1} = 0 . \end{equation}
Every solution to \eqref{d3} is periodic and may be written (up to translation in $\eta$) as,
\begin{equation} \label{d4} w(\eta ,a ) =a W\left( a^{-\frac{1}{2}(1-p)}\eta \right),\ \ \ \forall \eta \in \field{R}, \end{equation}
where $a \in \field{R}^+$ is a parameter and $W:\field{R}\to\field{R}$ is that unique periodic function which satisfies the problem,
\begin{equation} \label{d5'} W'' + W|W|^{p-1} = 0,\ \ \ \zeta \in\field{R} \end{equation}
\begin{equation} \label{d5} W(0) = 1, \ \ \ W'(0)=0 . \end{equation}
The period of $W(\zeta )$ is given by 
\begin{equation} \label{DI4} T(p) = 2^{3/2} (1+p)^{1/2} \int_0^1 \frac{d\lambda}{(1-\lambda^{(1+p)})^{1/2}} \end{equation}
whilst,
\begin{equation} \label{DI3} W(\zeta ) = -W\left( \frac{1}{2} T(p) - \zeta \right) = W(-\zeta ) \ \ \ \forall \zeta \in\field{R} . \end{equation}
Via an integration, the solution to \eqref{d5'}-\eqref{d5} satisfies 
\[ \frac{(W'(\eta ) )^2}{2} + \frac{|W(\eta )|^{1+p}}{(1+p)} = \frac{1}{(1+p)} \ \ \ \forall \eta \in\field{R} , \]
which represents a periodic orbit in the $(W,W')$ phase plane, as illustrated in Figure \ref{figd2}. It follows from \eqref{d4} that $w(\eta , a)$ has amplitude $a>0$ and period 
\begin{equation} \label{DI5} T_a(p) = a^{\frac{1}{2}(1-p)} T(p) . \end{equation}

\begin{figure}
  \centering
    \includegraphics[scale=0.3]{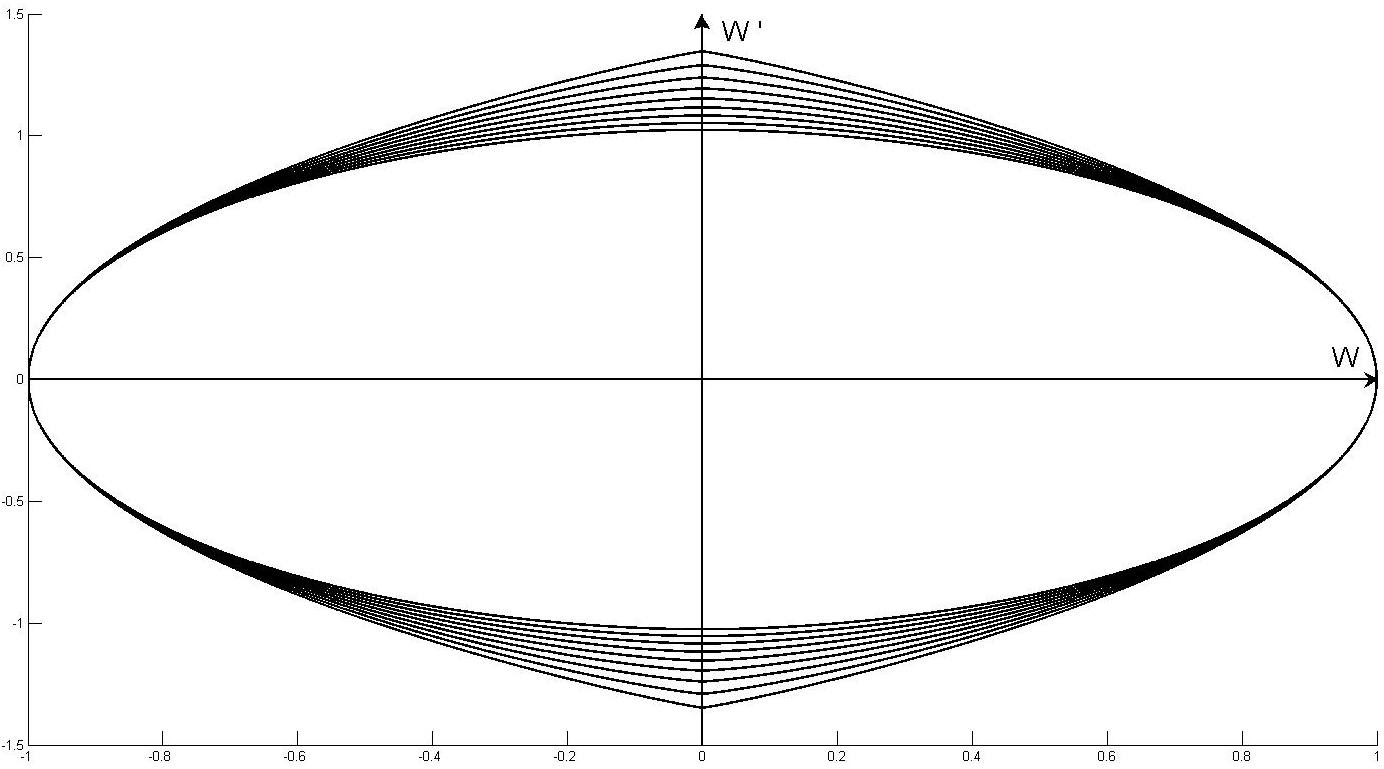}  \label{figd2}
		\caption{Phase paths for solutions to \eqref{d5'}-\eqref{d5} for $p_k=(0.1)k$ for $k=1...9$. Here the phase path for $p_k$ encloses the phase path for $p_{k+1}$ for $k=1...8$.}
\end{figure}
\noindent For any fixed $a\in\field{R}^+$, \eqref{d4} cannot represent the asymptotic structure to \eqref{DI1} and \eqref{DI2} since $W$ is periodic. The remaining terms in \eqref{DI1} must induce decay as $\eta \to \infty$. However, we observe from \eqref{DI5} that the oscillations in $w(\eta ,a)$ becomes increasingly rapid as the amplitude $a\to 0^+$. This suggests that we seek the asymptotic structure of \eqref{DI1}-\eqref{DI2} as $\eta\to\infty$ in the form,
\begin{equation} \label{d7} w(\eta ) \sim a(\eta )W(a (\eta )^{-\frac{1}{2}(1-p)}\eta ) \text{ as } \eta \to\infty , \end{equation}
with $a (\eta )>0$ and,
\begin{equation} \label{d8} a(\eta ),a'(\eta ) \to 0 \text{ as } \eta \to \infty . \end{equation}
Now, the rate of change of amplitude of oscillation in \eqref{d7}, $a'(\eta)$, approaches zero as $\eta \to \infty$, whilst the frequency of oscillation becomes unbounded as $\eta \to\infty$. We can thus use an averaging approach to determine an evolution equation for the amplitude $a(\eta )$ as $\eta \to \infty$. We substitute \eqref{d7} into \eqref{z1} and make use of \eqref{d5'}. We then integrate the resulting ordinary differential equation over {\emph{one}} period of $W(\cdot )$, over which, we may hold $a$ fixed. We obtain the leading order amplitude equation as,
\begin{equation} \label{d9} a'' + \frac{1}{2}\eta a' - \frac{1}{(1-p)}a = 0 ,\ \ \  \eta >>1 , \end{equation}
\begin{equation} \label{d10} a(\eta ),a'(\eta )\to 0 \text{ as } \eta\to \infty . \end{equation}
The linear ordinary differential equation \eqref{d9} has two basis functions $a_+:\field{R}\to\field{R}$ and $a_-:\field{R}\to\field{R}$ which have 
\[ a_+ (\eta ) \sim \eta^{-\left( 1+\frac{2}{(1-p)}\right)} e^{-\frac{1}{4}\eta^2} , \ \ \ a_- (\eta ) \sim \eta^{\frac{2}{(1-p)}} \ \ \ \text{ as }\eta \to\infty .\]
It follows that 
\begin{equation} \label{d11} a(\eta )\sim A_\infty \eta^{-\left( 1+\frac{2}{(1-p)}\right)}e^{-\frac{1}{4}\eta^2}\ \ \ \text{ as } \eta\to\infty ,\end{equation}
with $A_\infty$ being a positive globally determined constant dependent, in general, on $\alpha$, $\beta$ and $p$. Thus, from \eqref{d7}, we have 
\begin{equation} \label{es1} w_{\alpha , \beta}(\eta )\sim a (\eta ) W(a (\eta )^{-\frac{1}{2}(1-p)} \eta ) \text{ as }\eta \to\infty , \end{equation}
with, $\alpha (\eta )$ having the asymptotic form \eqref{d11} as $\eta \to \infty$. The same argument leads to the same (up to the constant $A_\infty$) asymptotic structure as $\eta \to -\infty$. As a consequence of \eqref{d11} and \eqref{es1}, we anticipate that $w_{\alpha , \beta } (\eta )$ decays to zero at a Gaussian rate as $|\eta |\to \infty$, whilst oscillating about zero with a local frequency which increases at a Gaussian rate as $|\eta |\to\infty$. This indicates that, in fact, $w_{\alpha , \beta }\in L^q(\field{R})$ for any $q>0$.

\subsection{Localized Solutions to [CP]} \label{sec32}
Following Corollary \ref{cor5} and Theorem \ref{thm9}, for each $(\alpha ,\beta )\in  \bar{D}_{c^*(p)}'$, we have constructed a non-trivial, localized, global solution $u_{\alpha , \beta }:\field{R}\times [0,\infty ) \to\field{R}$ to [CP], namely,
\begin{equation} \label{xx1} u_{\alpha , \beta }(x,t) = \begin{cases} t^{\frac{1}{(1-p)}} w_{\alpha , \beta } \left(\frac{x}{t^{1/2}}\right) &, (x,t)\in\field{R}\times (0,\infty ) \\ 0 &, (x,t)\in \field{R}\times \{ 0 \} . \end{cases} \end{equation}
With this two parameter family of solutions to [CP], each solution is distinct, and is not a spatial translate of any other solution in the family. However, we observe that $u_{\alpha , \beta }(x-x_0,t)$ is also a global solution to [CP] for any fixed $x_0\in\field{R}$. A trivial calculation from \eqref{xx1} establishes that 
\begin{equation} \label{xx2} (u_{\alpha , \beta})_x(x,t) = t^{\frac{1}{(1-p)} - \frac{1}{2}} w_{\alpha , \beta}'\left(\frac{x}{t^{1/2}}\right) , \end{equation}
\begin{equation} \label{xx3} (u_{\alpha , \beta})_{t}(x,t) = \frac{1}{(1-p)}t^{\frac{1}{(1-p)} - 1} \left( w_{\alpha , \beta}\left( \frac{x}{t^{1/2}}\right) - \frac{1}{2}(1-p)\left( \frac{x}{t^{1/2}}\right)w_{\alpha , \beta}' \left( \frac{x}{t^{1/2}}\right) \right) ,\end{equation}
for $(x,t)\in \field{R}\times (0,\infty )$, whilst from \eqref{i1},
\begin{equation} \label{xx4} (u_{\alpha , \beta})_{xx}(x,t) = (u_{\alpha , \beta})_{t}(x,t) - ( u_{\alpha , \beta }|u_{\alpha , \beta}|^{p-1} )(x,t) , \end{equation}
for $(x,t)\in \field{R}\times (0,\infty )$. It then follows immediately from Theorem \ref{thm9} that, 
\[ (u_{\alpha , \beta })_x,\ (u_{\alpha , \beta})_{t},\ (u_{\alpha , \beta})_{xx} \to 0 \text{ as } t\to 0^+ \text{ uniformly for }x\in\field{R} ,\]
and so, in fact, 
\begin{equation} \label{xx5} u_{\alpha , \beta} \in L^\infty (\field{R} \times [0,T])\cap C(\field{R}\times [0,T]) \cap C^{2,1} (\field{R} \times [0,T]) .\end{equation}
It follows from \eqref{xx5} that for each $(\alpha , \beta ) \in  \bar{D}_{c^*(p)}'$, and any $\tau>0$, then \linebreak[4] $u_{\alpha , \beta}^\tau :\field{R}\times [0,\infty )\to\field{R}$ such that 
\[ u_{\alpha , \beta}^\tau (x,t ) =  \begin{cases} (t-\tau )^{\frac{1}{(1-p)}} w_{\alpha , \beta }\left( \frac{x}{(t-\tau )^{1/2}}\right) &, (x,t)\in \field{R}\times (\tau , \infty ) \\ 0 &, (x,t)\in\field{R}\times [0,\tau ] \end{cases} \]
is also a non-trivial, localized, global solution to [CP]. Finally, we observe, via Theorem \ref{thm44} that for each $q>(1-p)/2$, then $u_{\alpha , \beta }(\cdot , t)\in L^q (\field{R}$ for each $t\geq 0$. Moreover, \eqref{d11} and \eqref{es1} suggest that the localization is Gaussian in $x$ for each $t > 0$.

\section{Heteroclinic Connections} \label{het}
In this section we establish the existence of at least one heteroclinic connection for (S) from the equilibrium point $(-(1-p)^{1/(1-p)},0)$ to the equilibrium point $((1-p)^{1/(1-p)},0)$.
\subsection{Existence}
We first consider solutions to the problem (S) for $\eta \in [0,\infty )$ and which remain in the region $\Omega\subset \field{R}^2$, given as
\begin{equation} \label{1'} \Omega = \left\{ (x,y) : 0<x<(1-p)^{1/(1-p)} , y>0 \right\} \end{equation}
with boundary $\partial \Omega = \overline{\Omega}\ \backslash\  \Omega$. We also define the following subset of $\partial \Omega$, namely, 
\begin{equation} \label{2'} \partial \Omega_1 = \left\{ (x,y) : x=0, y>0 \right\} . \end{equation} 
Specifically, we consider (S) for $\eta \in [0,\infty )$ and demonstrate that there exists a solution $(x,y):[0,\infty )\to \overline{\Omega}$ with zero-value $(0 , \beta)\in \partial \Omega_1$ and which satisfies 
\begin{equation} \label{2'c} (x(\eta ), y(\eta ))\in \Omega \ \ \ \forall \eta \in (0,\infty ) ,\end{equation}
\begin{equation} \label{2'd} (x(\eta ), y(\eta ))\to ((1-p)^{1/(1-p)}, 0) \ \ \ \text{ as } \eta \to \infty . \end{equation}

\noindent To begin with, it is readily established that for each zero-value $(0,\beta )\in \partial \Omega_1$, then (S) has a local solution {{\mbox{$(x,y):[0,\delta ] \to\field{R}^2$}} (for some $\delta >0$). Moreover, $(x(\eta ),y(\eta ))\in \Omega$ for $\eta \in (0,\delta ]$, and $x(\eta )$ is monotone increasing whilst $y(\eta )$ is monotone decreasing, with {\mbox{$\eta \in (0,\delta ]$.}} It is then straightforward to establish that $(x(\eta ),y(\eta ))$ can be {\emph{uniquely}} continued beyond $\eta = \delta $ and must satisfy one of the following three possibilities:
\begin{enumerate}
\item[(i)] There exists $\eta_\beta >0$ such that $(x(\eta ),y(\eta ))\in \Omega$ for all $\eta \in (0, \eta_\beta )$ and $(x(\eta_\beta ),y(\eta_\beta ))=((1-p)^{\frac{1}{1-p}},y_\beta )$ with $0 < y_\beta < \beta $, whilst $x'(\eta_\beta )= y_\beta >0$, and so there exists $\epsilon_\beta >0$ such that $(x(\eta ),y(\eta ))\not\in \bar\Omega \cup (\{ 0 \}\times \field{R})$ for $\eta \in (\eta_\beta , \eta_\beta +\epsilon_\beta ]$.
\item[(ii)] There exists $\eta_\beta >0$ such that $(x(\eta ),y(\eta ))\in \Omega$ for all $\eta \in (0, \eta_\beta )$ and $(x(\eta_\beta ),y(\eta_\beta ))=(x_\beta ,0 )$ with $0 < x_\beta < (1-p)^{\frac{1}{1-p}}$, whilst $y'(\eta_\beta ) <0$ and so there exists $\epsilon_\beta >0$ such that $(x(\eta ),y(\eta ))\not\in \bar\Omega \cup (\{ 0 \} \times \field{R})$ for $\eta \in (\eta_\beta , \eta_\beta +\epsilon_\beta ]$.
\item[(iii)] $(x(\eta ),y(\eta ))\in \Omega$ for all $\eta \in (0,\infty )$ and $(x(\eta ),y(\eta ))\to ((1-p)^{\frac{1}{1-p}},0)$ as $\eta\to \infty$.
\end{enumerate}
Our aim now is to obtain a uniqueness result for (S) with zero-value in $\partial \Omega_1$, and from this a continuous dependence result. This is non-trivial, since $\bf{Q}$ in \eqref{5} is not locally Lipschitz continuous in any neighborhood of $(0,\beta )\in \partial \Omega_1$, and so standard uniqueness and continuous dependence theory fail to apply. To begin with, we provide a local a priori bound for any solution of (S) with zero-value $(0,\beta )\in \partial\Omega_1$. 

\begin{prop} \label{A1'}
Let $(x,y):[0,\eta_\beta ]\to\field{R}^2$ be any solution to (S) with zero-value {\mbox{$(0,\beta )\in \partial \Omega_1$}} and which satisfies either case (i) or (ii). Then,
\begin{equation} \label{63*} \eta_\beta > \min \left\{\left(\frac{2}{\beta}\right)\left( m_H + \sqrt{m_H^2 + \frac{(\beta )^2}{2}}\right),\  \frac{(1-p)^{1/(1-p)}}{\beta} \right\}= \eta^* \end{equation}
with
\begin{equation} \label{A1a'} m_H = \inf_{\lambda \in [0,(1-p)^{1/(1-p)}]} \hspace{-5mm} H(\lambda ) ,\end{equation}
\end{prop}

\begin{proof}
Let $(x,y):[0,\eta_\beta ]\to\field{R}^2$ be any solution to (S) with zero-value $(0,\beta )\in \partial \Omega_1$, and which satisfies either case (i) or case (ii). Suppose that $\eta_\beta \leq \eta^*$. Since $(x(\eta ),y(\eta ))\in \Omega$ for all $\eta \in (0,\eta_\beta )$, it follows from \eqref{4} that 
\begin{equation} \label{*} \beta + m_H\eta -\frac{\beta}{4}\eta^2 < y(\eta ) < \beta \ \ \ \forall \eta \in (0, \eta_\beta ] . \end{equation}
However, $\eta_\beta \leq \eta^*$ and so, via \eqref{*},
\begin{equation} \label{+} \frac{\beta}{2} < y(\eta ) < \beta \ \ \ \forall \eta \in (0,\eta_\beta ] . \end{equation}
An integration of \eqref{3} using \eqref{+}, then gives, 
\begin{equation} \label{1-} \frac{\beta \eta}{2} < x(\eta ) < \beta \eta \ \ \ \forall \eta \in (0,\eta_\beta ] . \end{equation}
It finally follows from \eqref{+} and \eqref{1-}, since $\eta_\beta \leq \eta^*$, that,
\[ x(\eta_\beta ) < \beta \eta_\beta\leq \beta \eta^* \leq (1-p)^{\frac{1}{1-p}},\ \ \  y(\eta_\beta ) > \frac{\beta}{2}, \]
and so $(x(\eta_\beta ),y(\eta_\beta )) \in \Omega$, which is a contradiction. We conclude that $\eta_\beta > \eta^*$, as required.
\end{proof}

\noindent Therefore, we have,

\begin{cor} \label{A1}
Let $(x,y):[0,\eta^*]\to\field{R}^2$ be a solution to (S) with zero-value $(0,\beta )\in \partial \Omega_1$ with $\eta^*$ given by \eqref{63*}. Then,
\[\frac{\beta\eta}{2} < x(\eta ) < (1-p)^{1/(1-p)} ,\ \ \  \frac{\beta}{2} < y(\eta ) < \beta \ \ \  \forall \eta \in [0,\eta^*] ,\]
\end{cor}

\begin{proof}
For cases (i) and (ii), the result follows from Proposition \ref{A1'}, with case (iii) following immediately. 
\end{proof}

\noindent The a priori bounds in Corollary \ref{A1}, allow us to establish the following local uniqueness result for (S) with zero-value $(0 , \beta )\in \partial \Omega_1$. The proof is based on the uniqueness argument in \cite{JAME1}.

\begin{prop} \label{A2}
The problem (S) with zero-value $(0,\beta )\in \partial \Omega_1$ has at most one solution on $[0,\eta^*]$, with $\eta^*>0$ given by \eqref{63*}.
\end{prop}

\begin{proof}
To begin, fix $(0,\beta )\in \partial \Omega_1$. Suppose that $(x,y),(x^*,y^*):[0,\eta^*]\to\field{R}^2$ are solutions to (S) with zero-value $(0,\beta )$. It follows from Corollary \ref{A1} that 
\begin{equation} \label{A2a} (x(\eta ), y(\eta ) ),(x^*(\eta ),y^*(\eta )) \in \bar{\Omega} \ \ \ \forall \eta \in [0, \eta^* ] ,\end{equation}
whilst from Corollary \ref{A1}, 
\begin{equation} \label{A2aa} |x(\eta ) - x^*(\eta )| < (1-p)^{1/(1-p)},\ \ \ |y(\eta ) - y^*(\eta )| < \beta  \ \ \ \forall \eta \in [0,\eta^* ] .\end{equation}
Additionally, we observe that for $(X,Y)\in [0,  (1-p)^{1/(1-p)}] \times [0,\beta ]$, then
\begin{equation} \label{A2cc} X + X^p + Y < \left( 2 + \beta^{1-p} \right) ( X + Y )^p , \end{equation}
since $0<p<1$. Now, via \eqref{3} and \eqref{4} respectively, we have,
\begin{equation} \label{A2b} |x(\eta ) - x^*(\eta )| \leq \int_0^\eta |y(s) - y^*(s)|ds  \end{equation}
\begin{equation} \label{A2bb} |y(\eta ) - y^*(\eta )| \leq \int_0^\eta \left( \frac{1}{(1-p)}|x(s) - x^*(s)| + |x(s) - x^*(s)|^p + \frac{s}{2}|y(s) - y^*(s)|\right)   ds \end{equation}
for all $\eta\in [0,\eta^*]$. We next introduce $v:[0,\eta^*]\to\field{R}$ as,
\begin{equation} \label{A2w} v(\eta ) = |x(\eta ) - x^*(\eta )|  + |y(\eta ) - y^*(\eta )| \ \ \ \forall \eta\in [0,\eta^* ]. \end{equation}
Therefore, via \eqref{A2a}-\eqref{A2w}, it follows that
\begin{align}
\nonumber v(\eta ) \leq & \int_0^\eta \left( \frac{1}{(1-p)}|x(s) - x^*(s)| + |x(s) - x^*(s)|^p + \left(\frac{s}{2} +1 \right)|y(s) - y^*(s)|\right) ds \\
\nonumber \leq & \int_0^\eta \frac{1}{(1-p)}\left(\frac{\eta^*}{2} +1\right) \left( |x(s) - x^*(s)| + |x(s) - x^*(s)|^p + |y(s) - y^*(s)|\right) ds \\
\label{A2c} \leq & \int_0^\eta \frac{1}{(1-p)}\left(\frac{\eta^*}{2} +1\right) (2 + \beta^{1-p}) \left( v(s) \right)^p ds
\end{align}
for all $\eta \in [0,\eta^*]$, where the final inequality is due to \eqref{A2aa} and \eqref{A2cc}. Also, via Corollary \ref{A1} and \eqref{63*}, $\eta^*$ is dependent on $p$ and $\beta$ only, and hence, it follows from \eqref{A2c} that 
\begin{equation} \label{A2ccc} v(\eta ) \leq \int_0^\eta K(p,\beta ) \left( v(s) \right)^p ds \end{equation}
for all $\eta \in[0,\eta^*]$, where the constant $K(p,\beta )$ is given by,
\[ K(p,\beta ) =\frac{1}{(1-p)}\left(\frac{\eta^*}{2} +1\right) (2 + \beta^{1-p}) .\]
We now introduce the function $\bar{H}:[0,\eta^*]\to\overline{\field{R}}_+$ given by,
\begin{equation} \label{A2d} \bar{H}(\eta ) = \int_0^\eta K(p,\beta ) \left( v(s) \right)^p ds \ \ \ \forall \eta\in [0,\eta^* ].\end{equation}
It follows from \eqref{A2d} that $\bar{H}$ is non-negative, non-decreasing and differentiable on $[0,\eta^*]$, and via \eqref{A2ccc}, satisfies 
\begin{equation} \label{A2e} (\bar{H}(s))' \leq K(p,\beta )(\bar{H}(s))^p \ \ \ \forall s\in [0,\eta^* ] .\end{equation}
Upon integrating \eqref{A2e} from $0$ to $\eta$, we obtain
\begin{equation} \label{A2f}  \bar{H}(\eta ) \leq ((1-p)K(p,\beta )\eta )^{1/(1-p)} \ \ \ \forall \eta\in[0,\eta^* ] \end{equation}
and it follows from \eqref{A2f}, \eqref{A2d} and \eqref{A2ccc} that 
\begin{equation} \label{A2ff} v(\eta ) \leq \delta \ \ \ \forall \eta \in \left[ 0, \eta_\delta \right],\end{equation}
where $\delta >0$ is chosen sufficiently small so that
\[  \eta_\delta =  \frac{\delta^{1-p}}{ (1-p)K(p,\beta )} < \eta^* .\]
Now, from Corollary \ref{A1}, we have 
\begin{equation} \label{A2gg} \min \{x^* (\eta ) , x(\eta ) \} \geq \frac{\beta\eta}{2} \ \ \ \forall \eta \in [0, \eta^* ] .\end{equation}
Moreover, it follows from \eqref{5}, \eqref{A2gg} and the mean value theorem, that there exists $\theta (s) \geq \min \{x^*(s) , x(s)\}$, for which,
\begin{align}
\nonumber |Q_2(x(s ) &, y(s ) , s )  - Q_2(x^*(s ), y^*(s ) , s )|  \\
\nonumber & \leq \frac{1}{(1-p)}|x(s) - x^*(s)| + |x(s)^p - x^*(s)^p| +  \frac{ s}{2}  |y(s) - y^*(s)| \\
\nonumber & \leq \frac{1}{(1-p)}|x(s) - x^*(s)| + p(\theta ( s ))^{p-1}|x(s) - x^*(s)| + \frac{ \eta^*}{2} |y(s) - y^*(s)| \\
\nonumber & \leq \left( \frac{1}{(1-p)} + p\left( \frac{\beta s }{2}\right)^{p-1} \right) |x(s) - x^*(s)| +  \frac{\eta^*}{2} |y(s) - y^*(s)| \\
\label{A2g} & \leq \left( \frac{1}{(1-p)} + p\left(\frac{\beta s }{2}\right)^{p-1} + \frac{\eta^*}{2} \right) v(s ) 
\end{align}
for all $s\in (0,\eta^*]$. Now, via \eqref{3}, \eqref{4}, \eqref{5}, \eqref{A2cc}, \eqref{A2g} and \eqref{A2ff}, we have,
\begin{align}
\nonumber v(\eta ) \leq & \int_0^\eta \left( |Q_1(x(s) , y(s) , s )  - Q_1(x^*(s), y^*(s) , s)| \right. \\
\nonumber &\ \ \ \ \  + \left. |Q_2(x(s) , y(s) , s )  - Q_2(x^*(s ), y^*(s ) , s )| \right) ds \\
\nonumber \leq & \int_0^{\eta_\delta} K(p,\beta ) (v(s))^p ds + \int_{\eta_\delta}^\eta \left( 1 + \frac{1}{(1-p)} + p\left(\frac{\beta s }{2}\right)^{p-1} + \frac{\eta^*}{2} \right) v(s)ds \\
\label{A2h} \leq & \frac{\delta}{ (1-p)}  + \int_{\eta_\delta}^\eta \left( 1 + \frac{1}{(1-p)} + p\left(\frac{\beta s }{2}\right)^{p-1} + \frac{\eta^*}{2} \right) v(s ) ds
\end{align}
for all $\eta \in [\eta_\delta , \eta^* ]$. An application of Gronwall's Lemma \cite[Corollary 6.2]{HA1} to \eqref{A2h}, gives
\begin{equation} \label{A2i} v(\eta ) \leq \frac{\delta}{ (1-p)} e^{\left(  \eta^*\left( 1 + \frac{1}{(1-p)} + \left(\frac{\beta \eta^* }{2}\right)^{p-1} + \frac{\eta^*}{2} \right) \right)} \end{equation}
for all $\eta\in [\eta_\delta , \eta^*]$. Since $v$ is non-negative and $\eta^*$ is independent of $\delta$, it follows from \eqref{A2i} and \eqref{A2ff}, upon letting $\delta \to 0$, that 
\begin{equation} \label{A2j} v(\eta ) = 0 \ \ \ \forall \eta \in [0,\eta^*]. \end{equation}
Finally, it follows from \eqref{A2j} and \eqref{A2w} that 
\[ (x(\eta ),y(\eta ))= (x^*(\eta ),y^*(\eta )) \ \ \ \forall \eta \in [0, \eta^* ], \]
as required.
\end{proof}

\noindent We can now state the following uniqueness result.

\begin{lem} \label{lem3.4}
For each $(0,\beta )\in \partial \Omega_1$ then (S) with zero-value $(0,\beta )$ has exactly one solution $(x,y):I\to\field{R}^2$. This solution satisfies exactly one of the cases: (i) (with $I=[0,\eta_\beta +\epsilon_\beta ]$), (ii) (with $I=[0,\eta_\beta +\epsilon_\beta ]$) or (iii) (with $I=[0,\infty )$).
\end{lem}

\begin{proof}
We have established earlier that for each $(0,\beta )\in \partial \Omega_1$, then (S) with zero-value $(0,\beta )$ has at least one solution $(x,y):I\to\field{R}^2$, and that the solution satisfies one of the cases (i)-(iii). It follows from Proposition \ref{A2} that this solution is unique for $\eta \in [0,\eta^* ]$, (with $\eta^*$ depending only upon $\beta$ and $p$) and, moreover, in whichever case of (i)-(iii) it falls, that $(x(\eta ),y(\eta ))\not\in \{ (0, \lambda ):\lambda \in \field{R} \}$ for any $\eta \in I\ \backslash\ [0,\eta^*]$. Repeated application of the classical uniqueness theorem \cite[Chapter 1, Theorem 2.2]{CODLEV1} then completes the uniqueness result for $\eta \in I\ \backslash\ [0,\eta^*]$.
\end{proof}

\noindent We immediately obtain a continuous dependence result for solutions of (S) with zero-value in $\partial \Omega_1$, namely,

\begin{cor} \label{cor3.6}
Let $(0,\beta^*)\in \partial \Omega_1$ and suppose that the unique solution to (S) with zero-value $(0,\beta^*)$, say {\mbox{$(x^*,y^*):I\to\field{R}$,}} satisfies case (i) or (ii), with $I=[0,\eta_{\beta^*} + \epsilon_{\beta^*} ]$. Then, given $\epsilon ' >0$, there exists $\delta ' >0$ such that for all $\beta >0$ satisfying $|\beta - \beta^*|<\delta '$, the corresponding unique solution to (S) with zero-value $(0,\beta )$, say $(x,y):I'\to\field{R}$, has $I'=I$ and satisfies the corresponding case (i) or (ii), with, 
\[ |x(\eta ) - x^*(\eta )| + |y(\eta ) - y^*(\eta )| < \epsilon ' \ \ \ \forall \eta \in I. \]
\end{cor}

\begin{proof}
We first recall that (for a suitable choice of $\beta^*$) then 
\[ |x^*(\eta )| \leq (1-p)^{\frac{1}{1-p}}+1,\ \ \ |y^*(\eta)| \leq \beta^*+1 \ \ \ \forall \eta \in [0,\eta_{\beta^*} +\epsilon_{\beta^*} ] ,\]
and, via \eqref{5}, that ${\bf{Q}}(x,y,\eta )$ is continuous (and therefore bounded) on the rectangle 
\[ R=\left\{ (x,y,\eta ) : |x|\leq (1-p)^{\frac{1}{1-p}} +1,\ \ \ |y| \leq \beta^*+1,\ \ \ 0 \leq \eta \leq \eta_{\beta^*} +\epsilon_{\beta^*} \right\} . \]
The {\emph{uniqueness result}} in Lemma \ref{lem3.4} then allows for an application of the result \cite[Theorem 4.3, pp. 59]{CODLEV1} which completes the proof.
\end{proof}

\noindent It is now convenient to introduce the three sets $E_1$, $E_2$ and $E_3$, where
\[ E_1=\{ (0,\beta )\in \partial \Omega_1  \text{: the unique solution to (S)}\ \ \  \] \[\ \ \ \ \ \ \  \text{with zero-value }(0,\beta ) \text{ satisfies case (i)} \} ,\]
with $E_2$ and $E_3$ defined similarly for cases (ii) and (iii) respectively. It follows from {\mbox{Lemma \ref{lem3.4}}} that
\begin{equation} \label{dduu} E_i\cap E_j = \emptyset\ \text{ for }\ i,j=1,2,3 \text{ with } i\not= j ,\end{equation}
whilst
\begin{equation} \label{dduul} E_1\cup E_2 \cup E_3 = \partial \Omega_1 . \end{equation}
We now establish that $E_1$ and $E_2$ are both nonempty.

\begin{prop} \label{A8}
The set $E_1$ is non-empty and is such that $(0,\beta )\in E_1$ for each  
\begin{equation} \label{A8d} \beta > \sqrt{2\left( ((1-p)^{1/(1-p)} - m_H)^2 - m_H^2 \right)}  ,\end{equation}
with $m_H$ given by \eqref{A1a'}.
\end{prop}

\begin{proof}
Let $(x,y):I\to\field{R}^2$ be the unique solution to (S) with zero-value $(0,\beta )\in \partial \Omega_1$ and $\beta$ satisfying \eqref{A8d}. Since $(x(\eta ),y(\eta ))\in \bar{\Omega}$ for all $\eta \in I'$ (where $I'=[0,\eta_\beta ]$ for cases (i) and (ii), and $I'=[0,\infty )$ for case (iii)) then, via \eqref{3} and \eqref{4}, we have,
\begin{equation} \label{vv} \frac{\beta}{2} \leq y(\eta ) \leq \beta , \ \ \ x(\eta ) \geq \frac{\beta \eta }{2}\ \ \ \forall \eta \in [0,\bar{\eta}_\beta ], \end{equation}
with,
\begin{equation} \label{v} \bar{\eta}_\beta = \begin{cases} \min \{ \eta_\beta , \eta_\beta '\} &: \text{ cases (i) and (ii)} \\ \eta_\beta ' &: \text{ case (iii) } \end{cases} \end{equation}
and
\[ \eta_\beta ' = \frac{2}{\beta}\left( m_H + \sqrt{m_H^2 + \frac{\beta^2}{2}}\right). \]
Now suppose case (iii) occurs, then $(x(\eta_\beta ' ),y(\eta_\beta '))\in \Omega$. However, 
\[ x(\eta_\beta ') \geq \frac{\beta \eta_\beta '}{2} =  m_H + \sqrt{m_H^2 + \frac{\beta^2}{2}} > (1-p)^{\frac{1}{1-p}} \]
via \eqref{vv} and \eqref{A8d}, and we arrive at a contradiction. We can therefore eliminate case (iii). Next suppose case (ii) occurs. It follows from \eqref{vv}$_2$ and \eqref{v} that $\eta_\beta \leq \eta_\beta '$, and so $\bar{\eta}_\beta = \eta_\beta$. Thus, via \eqref{vv}$_1$,
\[ y(\eta_\beta ) \geq \frac{\beta}{2} > 0. \]
However, in case (ii), $y(\eta_\beta ) = 0$, and we arrive at a contradiction. We conclude finally that case (i) must occur, as required. 
\end{proof}

\noindent We can also establish a similar result for $E_2$.

\begin{prop} \label{A7}
The set $E_2$ is non-empty and is such that $(0,\beta )\in E_2$ for each
\begin{equation} \label{A7a}  0< \beta < \sqrt{\frac{(1-p)^{2/(1-p)}}{(1+p)}}. \end{equation}
\end{prop}

\begin{proof}
It follows from \eqref{7}-\eqref{13} that for $\beta$ satisfying the inequality \eqref{A7a}, then {\mbox{$(0,\beta )\in D_{c^*(p)}$.}} It then follows from Corollary \ref{cor5} that (S) with zero-value $(0,\beta )$ has a global solution which lies in $D_{c_\beta}$ for all $ \eta \in (0,\infty )$ with $c_\beta = V(0,\beta ) < c^*(p)$, and so the solution to (S) in $\eta \geq 0$ must satisfy case (ii). Therefore, $(0,\beta )\in E_2$, as required. 
\end{proof}

\noindent We next establish that both $E_1$ and $E_2$ are open subsets of $\partial \Omega_1$. 

\begin{prop} \label{proper}
The sets $E_1$ and $E_2$ are open subsets of $\partial \Omega_1$. 
\end{prop}

\begin{proof}
We will prove the result for $E_1$. The proof for $E_2$ is similar. Let $(0,\beta^*)\in E_1$. Then, via Lemma \ref{lem3.4}, (S) with zero-value $(0,\beta^*)$ has a unique solution {\mbox{$(x^*,y^*):[0,\eta_{\beta^*} + \epsilon_{\beta^*}]\to\field{R}^2$,}} with
\begin{equation} \label{n} (x^*(\eta ) ,y^*(\eta )) \in \Omega \ \ \ \forall \eta \in (0,\eta_{\beta^*}) \end{equation}
and
\begin{equation} \label{nn} (x^*(\eta_{\beta^*} ) ,y^*(\eta_{\beta^*} )) =((1-p)^{1/(1-p)},y_{\beta^*}) \end{equation}
for some $0 < y_{\beta^*}< \beta^*$, whilst
\begin{equation} \label{nnn} (x^*(\eta ) ,y^*(\eta )) \not\in \bar{\Omega} \ \ \ \forall \eta \in (\eta_{\beta^*}, \eta_{\beta^*}+ \epsilon_{\beta^*} ] . \end{equation}
Now consider the family of open balls 
\[ B(x^*(\eta ),y^*(\eta ) ;\epsilon ') \text{ with } \eta \in [0,\eta_{\beta^*} +\epsilon_{\beta^*} ] \]
and via \eqref{n}-\eqref{nnn}, choose $\epsilon '$ sufficiently small so that 
\begin{equation}\label{x} B(x^*(\eta_{\beta^*} + \epsilon_{\beta^*}), y^*(\eta_{\beta^*} + \epsilon_{\beta^*}); \epsilon ') \cap \bar{\Omega} = \emptyset \end{equation}
and
\begin{equation}\label{xx} \bigcup_{\lambda\in [0,\eta_{\beta^*}+\epsilon_{\beta^*}]} B (x^*(\lambda ), y^*(\lambda ); \epsilon ') \cap (\partial \Omega\ \backslash\ \partial \Omega_1) \subset \{ ((1-p)^{1/(1-p)},\lambda ) : \lambda >0 \} . \end{equation}
It then follows from Corollary \ref{cor3.6} that there exists $\delta ' >0$ such that the corresponding unique solution to (S) with zero-value $(0,\beta )\in \partial \Omega_1$, satisfying  $|\beta - \beta^*| < \delta '$, say {\mbox{$(x,y):[0,\eta_{\beta^*} + \epsilon_{\beta^*}]\to\field{R}^2$}} has 
\begin{equation} \label{xxx} (x(\eta ), y(\eta ))\in \bigcup_{\lambda\in [0,\eta_{\beta^*}+\epsilon_{\beta^*}]} B(x^*(\lambda ), y^*(\lambda ); \epsilon ') \ \ \ \forall \eta \in [0,\eta_{\beta^*} + \epsilon_{\beta^*}] \end{equation}
Therefore, via \eqref{x}-\eqref{xxx}, $\{ (0,\beta ): |\beta - \beta^*|< \delta ' \} \subseteq E_1$, and so $E_1$ is an open subset of $\partial \Omega_1$, as required.  
\end{proof}
 
\noindent Finally, we have 

\begin{cor} \label{corrie!}
The set $E_3$ is a non-empty closed subset of $\partial \Omega_1$.
\end{cor}

\begin{proof}
Via Propositions \ref{A8} and \ref{A7}, $E_1$ and $E_2$ are both nonempty subsets of $\partial \Omega_1$. Moreover, via \eqref{dduu} $E_1$ and $E_2$ are disjoint. Suppose that $E_3$ is empty, then via \eqref{dduul} and Proposition \ref{proper}, $E_1$ and $E_2$ form an open partition of $\partial \Omega_1$. However, $\partial \Omega_1$ is a connected subset of $\field{R}^2$, and we arrive at a contradiction. Hence $E_3$ must be nonempty. Finally, $E_3=\partial \Omega_1 \backslash (E_1 \cup E_2)$ and is therefore a closed subset of $\partial \Omega_1$.  
\end{proof}

\begin{rmk}
In Corollary \ref{corrie!}, the existence of at least one point in $E_3$ has been established. However, it has not been established that this is the only point in $E_3$. 
\end{rmk}

\noindent To conclude this section, we arrive at our main result, namely,

\begin{thm} \label{mainowino}
There exists a solution $(x,y):\field{R}\to \field{R}^2$ to (S) with zero-value $(0,\beta )\in \partial \Omega_1$, for some 
\[ \sqrt{\frac{(1-p)^{2/(1-p)}}{(1+p)}} \leq \beta \leq \sqrt{2\left( ((1-p)^{1/(1-p)} - m_H)^2 - m_H^2 \right)} , \]
which satisfies
\begin{equation} \label{XOman} (x(\eta ),y(\eta )) \to (\pm (1-p)^{1/(1-p)},0) \text{ as } \eta \to \pm\infty \end{equation}
and
\begin{equation} \label{XOman'} |x(\eta ) | < (1-p)^{1/(1-p)},\ \ \ 0 < y(\eta ) \leq \beta \ \ \ \forall \eta \in \field{R}. \end{equation}
\end{thm}

\begin{proof}
It follows directly from Corollary \ref{corrie!} and (iii) that there exists $(x^*,y^*):[0,\infty )\to\field{R}^2$ which is a solution to (S) with zero-value $(0,\beta^*)$, such that 
\begin{align}
\label{XOman1} & (x^*(\eta ),y^*(\eta ))\to ((1-p)^{1/(1-p)},0) \text{ as } \eta \to \infty , \\
\label{117*} &  (x^*(\eta ),y^*(\eta )) \in \Omega \ \ \ \forall \eta \in (0,\infty ) . 
\end{align}
It follows from \eqref{A7a} and \eqref{A8d}, that  
\[  \sqrt{\frac{(1-p)^{1/(1-p)}}{(1+p)}}\leq \beta^* \leq \sqrt{2\left( ((1-p)^{1/(1-p)} - m_H)^2 - m_H^2 \right)} .\]
Now, define the function $(x,y):\field{R}\to\field{R}^2$ to be
\begin{equation} \label{XOman1'} (x(\eta ), y(\eta )) = \begin{cases} (x^*(\eta ),y^*(\eta )) &; \eta \in [0,\infty ) \\ (-x^*(-\eta ),y^*(-\eta )) &; \eta \in (-\infty , 0) .\end{cases}  \end{equation}%
It follows from \eqref{XOman1'} that $(x,y):\field{R}\to\field{R}^2$ is a solution to (S) with zero-value $(0,\beta^* )$, and via \eqref{1'} and (iii), (since $y (\eta )$ is monotone decreasing for $\eta \in (0,\infty )$) that this solution satisfies \eqref{XOman} and \eqref{XOman'}.
\end{proof}

We conclude from Theorem \ref{mainowino} that the problem (S) has at least one heteroclinic connection from the equilibrium point $(-(1-p)^{1/(1-p)},0)$ ($\eta = -\infty$) to the equilibrium point $((1-p)^{1/(1-p)},0)$ ($\eta = \infty$), which we denote by $w_{\beta^*}:\field{R}\to\field{R}$. Here $w=w_{\beta^*}(\eta )$, $\eta \in \field{R}$, has zero-value $w(0)=0$, $w'(0)=\beta^*$ for some 
\[ \sqrt{\frac{(1-p)^{2/(1-p)}}{(1+p)}} \leq \beta^* \leq \sqrt{2\left( ((1-p)^{1/(1-p)} - m_H)^2 - m_H^2 \right)} , \]
and 
\[ |w_{\beta^*}(\eta )| < (1-p)^{1/(1-p)},\ \ \ 0 < w_{\beta^*}'(\eta ) \leq \beta^* \ \ \ \forall \eta \in \field{R},\]
recalling also, that $w_{\beta^*}(\eta )$ is an odd function of $\eta\in\field{R}$. Finally, a straightforward linearization as $|\eta |\to\infty$ establishes that, 
\[ w_{\beta^*}(\eta ) \sim \pm (1-p)^{1/(1-p)} - \frac{A_\infty}{\eta^3}e^{-\frac{1}{4}\eta^2}\ \text{ as }\ \eta \to \pm \infty ,\]
with $A_\infty$ being a globally determined constant.

\subsection{Front Solutions to [CP]}  \label{sec42}
Following Theorem \ref{mainowino}, with $\beta = \beta^*$ we have constructed the front-like global solution $u_{\beta^*}:\field{R}\times [0,\infty )\to\field{R}$ to [CP], namely, 
\begin{equation} \label{noidea} u_{\beta^*}(x,t) = \begin{cases} t^{\frac{1}{(1-p)}} w_{\beta^*} \left( \frac{x}{t^{1/2}}\right) &, (x,t)\in \field{R}\times (0,\infty ) \\ 0 &, (x,t)\in \field{R} \times \{ 0 \} . \end{cases} \end{equation}
We again observe that $u_{\beta^*}(x-x_0,t)$ is also a global solution to [CP] for any fixed $x_0\in\field{R}$. In addition, following Section 3.3, we conclude that, for any $\tau >0$, $u_{\beta^*}^\tau :\field{R}\times [0,\infty )\to\field{R}$ such that
\[ u_{\beta^*}^\tau (x,t) = \begin{cases} (t-\tau )^{\frac{1}{(1-p)}} w_{\beta^*}\left( \frac{x}{(t-\tau )^{1/2}} \right) &, (x,t) \in\field{R} \times (\tau , \infty ) \\ 0 &, (x,t) \in \field{R}\times [0,\tau ] \end{cases} \]
is also a front-like global solution to [CP].

\section{Discussion} \label{sec5}
There are two questions that arise naturally from this study. The first being how one can rigorously establish the decay rate of the homoclinic solutions $w:\field{R}\to\field{R}$ to (S) as $\eta\to\pm\infty$, that is suggested by \eqref{d11} and \eqref{es1}; the second being whether or not for the problem (S), there is a unique heteroclinic connection from the equilibrium point $(-(1-p)^{1/(1-p)},0)$ to the equilibrium point $((1-p)^{1/(1-p)},0)$ which has zero value in $\partial \Omega_1$ (Theorem \ref{mainowino} guarantees that there exists at least one connection).

 % \bibliographystyle{plain} 
 % \bibliography{mybib}

\newcommand{\Addresses}{{% additional braces for segregating \footnotesize
  \bigskip
  \footnotesize

  J. C.~Meyer, \textsc{School of Mathematics, Watson Building, University of Birmingham, Birmingham, UK, B15 2TT}\par\nopagebreak
  \textit{E-mail address}: \texttt{J.C.Meyer@bham.ac.uk}

  \medskip

  D. J.~Needham, \textsc{School of Mathematics, Watson Building, University of Birmingham, Birmingham, UK, B15 2TT}\par\nopagebreak
  \textit{E-mail address}: \texttt{D.J.Needham@bham.ac.uk}

	}}
	
	\Addresses

\end{document}